\documentclass[11pt]{amsart}
\usepackage{amsmath,amsthm,amssymb,enumerate,mathscinet,mathtools}
\usepackage{fullpage}
\usepackage{graphicx,xcolor,pgfplots}
\usepackage{verbatim}
\usepackage{mathrsfs}
\usepackage{hyperref}
\usepackage{enumitem}
\usepackage{blindtext}
\usepackage{multicol}
\newcommand{\fm}{f_{\max}}

\newtheorem{theorem}{Theorem}[section]

\newtheorem{lemma}[theorem]{Lemma}
\newtheorem{question}[theorem]{Question}
\newtheorem{proposition}[theorem]{Proposition}

\newtheorem{claim}[theorem]{Claim}
\newtheorem{definition}[theorem]{Definition}

\newtheorem{thm}[theorem]{Theorem}
\newtheorem{prop}[theorem]{Proposition}
\newtheorem{example}[theorem]{Example}

\newtheorem{conj}[theorem]{Conjecture}

\numberwithin{equation}{section}

\def\N{\mathbb{N}}
\def\Z{\mathbb{Z}}

\newcommand{\mis}{\textsc{mis}}
\makeatletter
\newcommand*{\rom}[1]{\expandafter\@slowromancap\romannumeral #1@}
\makeatother

\def\COMMENT#1{}
\let\COMMENT=\footnote

\allowdisplaybreaks

\title{On maximal sum-free sets in abelian groups}

\author{Nathana\"el Hassler and
 Andrew Treglown}

\thanks{NH: Ecole Normale Sup\'erieure (ENS) de Rennes, France,
{\tt nathanael.hassler@ens-rennes.fr}
AT: University of Birmingham, United Kingdom, {\tt a.c.treglown@bham.ac.uk}, research supported by EPSRC grant EP/V002279/1.}

\begin{document}
 
\begin{abstract}
Balogh, Liu, Sharifzadeh and Treglown [Journal of the European Mathematical Society, 2018] recently gave a sharp count on the number of maximal sum-free subsets of $\{1, \dots, n\}$, thereby answering a question of  Cameron and Erd\H{o}s. In contrast, not as much is know about the analogous problem for finite abelian groups. In this paper we give the first sharp results in this  direction, determining asymptotically the number of 
maximal sum-free sets in both the binary and ternary spaces $\mathbb Z^k_2$ and $\mathbb Z^k_3$.  We also make  progress on a conjecture of Balogh, Liu, Sharifzadeh and Treglown concerning a general lower bound on the number of maximal sum-free sets in abelian groups of a fixed order. Indeed, we verify the conjecture for all finite abelian groups with a cyclic component of size at least 3084. Other related results and open problems are also presented. 
\end{abstract}

\date{\today}

\maketitle

\section{Introduction}
Let $G$ be a group or set of integers. A triple $x,y,z \in G$ is a \emph{Schur triple} if $x+y=z$ (note $x,y$ and $z$ may not necessarily be distinct). We say that a subset $A$ of $G$ is \emph{sum-free} if $A$ contains no Schur triples. A \emph{maximal sum-free} subset of $G$ is a sum-free set $A \subseteq G$ that is not properly contained in another sum-free subset of $G$.

There has been significant interest in the study of sum-free subsets of $[n]:=\{1,\dots,n\}$. It is an easy exercise to prove that the largest sum-free subset of $[n]$ has size $\lceil n/2 \rceil$ (attained, e.g., by the set of odd integers in $[n]$). Answering  (a stronger version of) a conjecture of Cameron and Erd\H{o}s~\cite{cam1},  Green~\cite{G-CE} and independently Sapozhenko~\cite{sap} proved the following:  there are constants $C_1$ and $C_2$ such that the number of sum-free subsets in $[n]$ is $(C_i+o(1))2^{n/2}$ 
for all $n \equiv i \pmod{2}$. Addressing another problem of Cameron and Erd\H{o}s~\cite{CE},  Balogh, Liu, Sharifzadeh and the second author~\cite{BLST2} proved the following: for each $1\leq i \leq 4$, there is a constant $D_i$ such that, given any $n\equiv i \pmod{4}$,
$[n]$ contains  $(D_i+o(1)) 2^{n/4}$ maximal sum-free sets.

There has also been interest in studying analogous questions in the setting of groups.
In what follows
 $G$ will always be a finite abelian group, where unless stated otherwise, $G$ has order $n$. Denote by $\mu(G)$ the size of a largest sum-free subset of $G$. Denote by $f(G)$ the number of sum-free subsets of $G$ and by $\fm(G)$ the number of maximal sum-free subsets of $G$. 

The study of sum-free sets in finite abelian groups dates back to the 1960s. Indeed, Diananda and Yap~\cite{yan} determined $\mu (G)$ for all so-called type \rom{1} and \rom{2} groups. However, it was not until 2005 that Green and Ruzsa~\cite{GR-g} determined $\mu(G)$ exactly for all finite abelian groups; see Theorem~\ref{mu(G)} below. In particular, for every finite abelian group $G$ of order $n$, $2n/7 \leq \mu (G) \leq n/2$.
Further, Green and Ruzsa~\cite{GR-g} determined $f(G)$ up to an error term in the exponent for all finite abelian groups $G$, showing that $f(G)=2^{\mu(G)+o(n)}$. A
refined version of this counting result for type \rom{1} groups was obtained by  Alon,  Balogh,  Morris and  Samotij~\cite{alon}. 

Much less is known, however, about the value of 
$\fm(G)$. Improving on an earlier bound of Wolfovitz~\cite{wolf}, the following result provides a general upper bound on $\fm(G)$ and  shows there is an exponential gap between the values of $f(G)$ and $\fm(G)$.

\begin{prop}  \cite{BLST2} \label{easyb} Let $G$ be an abelian group of order $n$.
Then
\begin{align*}
\fm(G)\le 3^{\mu(G)/3+o(n)}.
\end{align*}
\end{prop}

 Let $\mathbb Z_p^k:=\mathbb Z_p\times \mathbb Z_p\times\cdots\times \mathbb Z_p$. Sum-free sets in $\mathbb Z^k_2$ have been  well studied. As noted in the introduction of \cite{ped}, `sum-free sets [in $\mathbb Z^k_2$] also occur in geometry as blocking sets... and in coding theory as the set of columns of a parity check matrix for a linear code with minimum distance at least four'.
 Motivated by coding theory applications, Davydov and Tombak~\cite{dav} and also independently Clark and Pedersen~\cite{ped} provided a
 structural characterisation of `large' sum-free subsets of $\mathbb Z^k_2$ (see Theorem~\ref{structure Z2} below).
In~\cite{BLST2} it was proven that $\fm(\mathbb Z^k_2)=2^{{\mu(\mathbb Z^k_2)}/{2}+o(n)}$ where $\mu(\mathbb Z^k_2)=n/2$. Our first theorem gives a sharp version of this result.
\begin{thm}\label{main} Let $k \in \mathbb N$ and $n:=2^k$. Then
$$\fm (\mathbb Z^k_2)= \left (\binom{n-1}{2}+o(1) \right ) 2^{n/4}.$$
\end{thm}
Theorem~\ref{main} provides the first finite abelian groups for which we now do have a  sharp count on the number of maximal sum-free subsets.
We also believe Theorem~\ref{main} is of interest to the projective space  community. Indeed, the notion of a so-called complete cap in a projective space 
$PG(k,2)$ over $\mathbb F_2$ is equivalent to the notion of a maximal sum-free set in 
$\mathbb Z^{k+1}_2$. Thus, Theorem~\ref{main} asymptotically determines the number
of complete caps in $PG(k,2)$; see Section~\ref{sec:pro} for further details.

Other than for $\mathbb Z^k_2$, the only other finite abelian groups $G$ for which $\fm(G)$ is known up to an error term in the exponent are $G$ for which  $9|n$ but $p \nmid n$ for every prime $p$ with $p \equiv 2 \pmod{3}$. In this case, Liu and Sharifzadeh~\cite{ls} proved that the upper bound in Proposition~\ref{easyb} is in fact tight. That is, for such $G$ we have $\fm(G)=3^{{\mu(G)}/{3}+o(n)}$. 

Our next theorem sharpens this result when $G=\mathbb Z^k_3$; note that $\mu(\mathbb Z^k_3)=3^{k-1}$.
\begin{thm}\label{thm2}
Let $k\in\N$ and $n:=3^k$. Then
$$\fm(\Z_3^k)=\left(\frac{(n-3)(n-1)}{3}+o(1)\right) 3^{n/9}.$$
\end{thm}
The proofs of  Theorem~\ref{main} and Theorem~\ref{thm2} draw on a number of ideas and tools from previous papers~\cite{BLST, BLST2, ls} and in particular, utilise a container theorem of Green and Ruzsa~\cite{GR-g}. We also utilise structural results from~\cite{ped, dav, lev} (which view $ \mathbb Z_p^k$ as a vector space over $\mathbb Z_p$).
It is worth remarking that our proofs  actually show that the $o(1)$ terms in  Theorems~\ref{main} and~\ref{thm2} are  of the form $O(2^{-\delta n})$ for some fixed $\delta >0$.

In both~\cite{BLST2} and \cite{ls} a number of questions and conjectures concerning $\fm(G)$ are raised. In particular,  motivated by the fact that  $2n/7 \leq \mu (G) \leq n/2$ for every finite abelian group $G$ of order $n$, the following conjecture was made in~\cite{BLST2}.
\begin{conj}\cite{BLST2}\label{conj1}
For every abelian group $G$ of order $n$,
$$2^{n/7} \leq f_{\max} (G) \leq 2^{n/4+o(n)},$$
where the bounds, if true, are best possible.
\end{conj}
In this paper we give a wide class of abelian groups for which the lower bound in Conjecture~\ref{conj1} holds, including all abelian groups that have a cyclic component of size at least $3084$ (see Proposition~\ref{prop34} below).

\smallskip

The paper is organised as follows. In the next section we collect together some results that will be applied in our proofs. We prove Theorems~\ref{main} and~\ref{thm2} in Sections~\ref{sec:mainthm} and~\ref{sec:thm2} respectively. In Section~\ref{sec:construct} we give various constructions that provide  lower bounds on $\fm (G)$. Finally, we give some concluding remarks and open problems in Section~\ref{sec:con}.

\section{Useful results}\label{sec:use}

\subsection{Maximum size sum-free sets and containers}
The following definitions play an important role in describing the behaviour of $\mu(G)$
for finite abelian groups.
\begin{definition}
Let $G$ be an abelian group of order $n$.
\begin{itemize}
    \item Let $n$ be divisible by a prime $p\equiv 2 \pmod{3}$. Given the smallest such $p$, we  say that $G$ is \emph{type \rom{1}$(p)$}.
    \item If $n$ is not divisible by any prime $p\equiv 2\pmod{3}$, but $3|n$, then we say that $G$ is \emph{type \rom{2}}.
    \item Otherwise, $G$ is \emph{type \rom{3}}.
\end{itemize}
\end{definition}

As summarised in the following theorem,
results from \cite{yan,GR-g} determine
the size of a largest sum-free set for any abelian group. 

\begin{theorem}\label{mu(G)}
Given any finite abelian group $G$, if $G$ is type \rom{1}$(p)$ then $\mu(G)=|G|\left(\frac{1}{3}+\frac{1}{3p}\right)$. Otherwise, if $G$ is type \rom{2} then $\mu(G)=\frac{|G|}{3}$. Finally, if $G$ is type \rom{3} then $\mu(G)=|G|\left(\frac{1}{3}-\frac{1}{3m}\right)$ where $m$ is the exponent (largest order of any element) of $G$.
\end{theorem}

We will also use the following container result of Green and Ruzsa \cite{GR-g} (the version below is the one stated in \cite{ls}).
\begin{lemma}\cite{GR-g}\label{containers}
Given any finite abelian group $G$ of order $n$ there is a family $\mathcal{F}$ of subsets of $G$ with the following properties.
\begin{enumerate}[label = (\roman*),itemsep=0pt]
    \item Every $F\in\mathcal{F}$ has at most $(\log n)^{-1/9}n^2$ Schur triples.
    \item If $S\subseteq G$ is sum-free, then $S$ is contained in some $F\in\mathcal{F}$.
    \item $|\mathcal{F}|\leq 2^{n(\log n)^{-1/18}}$.
    \item Every member of $\mathcal{F}$ has size at most $\mu(G)+n(\log n)^{-1/50}$.
\end{enumerate}
\end{lemma}
We refer to the sets in $\mathcal{F}$ as \textit{containers}. The result gives us a method to count the number of maximal sum-free sets in an abelian group $G$. Indeed, by (ii) it suffices to count the number of maximal sum-free sets that lie in our containers.

\subsection{Structural results for sum-free sets}

As we wish to obtain sharp bounds on $\fm(\Z_2^k)$ and $\fm(\Z_3^k)$ we must be rather careful about how we count the maximal sum-free sets in each container. As such we need information on the structure of sum-free sets in $\Z_2^k$ and $\Z_3^k$. The following theorem  describes the structure of sum-free sets in $\Z_2^k$.

\begin{thm}\cite{ped, dav}\label{structure Z2}
Let $k\geq 4$ and let $A\subseteq \Z_2^k$ be a sum-free set. If $|A|>5\cdot 2^{k-4}$ then $A$ is contained in a coset $x+U$ where $U\subseteq\Z_2^k$ is a subspace of $\Z_2^k$ and $x\notin U$.
\end{thm}

The following result of Lev \cite{lev} provides a similar structure for sum-free sets in $\Z_3^k$.

\begin{thm}\cite{lev}\label{structure Z3}
Let $k\geq 3$ and let $A\subseteq \Z_3^k$ be a sum-free set. If $|A|>5\cdot 3^{k-3}$ then $A$ is contained in a coset $g+H$ where $H\subseteq\Z_3^k$ is a subspace of $\Z_3^k$ and $g\notin H$.
\end{thm}

Thus, both Theorems \ref{structure Z2} and \ref{structure Z3} state that either a sum-free set is `small' or has a simple structure; that is, it is contained in a coset.

\subsection{Maximal independent sets in graphs}
Similarly to previous papers on the topic~\cite{BLST, BLST2, ls, wolf}, we translate our problems into the setting of maximal independent sets in graphs. 
As such,
in this subsection we introduce some notation, definitions and useful results for graphs and their maximal independent sets. 

Let $\Gamma=(V,E)$ be a graph. Throughout this paper we consider \textit{graphs $\Gamma$ possibly with loops}; that is, $\Gamma$ can be obtained from a simple graph by adding at most one loop at each vertex. We write $e(\Gamma):=|E|$ for the number of edges in $\Gamma$. Given a vertex $x\in V$ we write $d(x,\Gamma)$ for the \textit{degree} of $x$ in $\Gamma$ (i.e., the number of times $x$ is an  endpoint of an edge in $G$). Note that a loop at $x$ contributes two to the degree of $x$. We write $\delta(\Gamma)$ for the \textit{minimum degree} and $\Delta(\Gamma)$ for the \textit{maximum degree} of $\Gamma$. We write $K_m$ and $C_m$ respectively for the complete graph and the cycle on $m$ vertices. Given graphs $\Gamma$ and $\Gamma'$ we write $\Gamma\square\Gamma'$ for the \textit{cartesian product graph}. So its vertex set is $V(\Gamma)\times V(\Gamma')$ and $(x,y)$ and $(x',y')$ are adjacent in $\Gamma\square\Gamma'$ if (i) $x=x'$ and $y$ and $y'$ are adjacent in $\Gamma'$ or (ii) $y=y'$ and $x$ and $x'$ are adjacent in $\Gamma$. 

We denote by $\mis(\Gamma)$ the number of maximal independent sets in $\Gamma$. Moon and Moser \cite{MM} proved the following bound which holds for any $n$-vertex graph $\Gamma$:
\begin{equation}\label{Moon-Moser}
\mis(\Gamma)\leq 3^{n/3}.
\end{equation}
Note the bound (\ref{Moon-Moser}) is tight; consider a graph consisting of the disjoint union of triangles. However, Hujter and Tuza \cite{HT} improved this bound in the case of  triangle-free graphs $\Gamma$:
\begin{equation}\label{Hujter-Tuza}
\mis(\Gamma)\leq 2^{n/2}.
\end{equation}
If $\Gamma$ is a perfect matching then we have equality in (\ref{Hujter-Tuza}). The following lemma improves the bound (\ref{Moon-Moser}) in the case of somewhat regular and dense graphs.

\begin{lemma}\cite[Equation (3)]{BLST}\label{dense and regular graphs}
Let $k\geq 1$ and let $\Gamma$ be a graph on $n$ vertices. Suppose that $\Delta(\Gamma)\leq k\delta(\Gamma)$ and set $b:=\sqrt{\delta(\Gamma)}$. Then
$$\mis(\Gamma)\leq \sum_{0\leq i\leq n/b}\binom{n}{i}\cdot3^{\left(\frac{k}{k+1}\right)\frac{n}{3}+\frac{2n}{3b}}.$$
\end{lemma}

We will also use the following refined version of the Moon--Moser bound (\ref{Moon-Moser}).

\begin{lemma}\cite[Lemma 3.5]{ls}\label{stability}
Let $k\in\Z$, $\Delta\in\N$, and $C\geq 3^{\Delta/13}$. Let $\Gamma$ be an $n$-vertex graph with $n+k$ edges and maximum degree $\Delta$, then
$$\mis(\Gamma)\leq C\cdot 3^{\frac{n}{3}-\frac{k}{13\Delta}}.$$
\end{lemma}

In order to connect our problem of estimating the number of maximal sum-free sets in a group $G$ to the count of the maximal independent sets in a graph, we define the \textit{link graph}. For subsets $B,S\subseteq G$ let $L_S[B]$ be the \textit{link graph of $S$ on $B$} defined as follows. The vertex set of $L_S[B]$ is $B$ and its edge set is composed of the following edges:
\begin{enumerate}[label = (\roman*),itemsep=0pt]
    \item two vertices $x,y\in B$ are adjacent if there exists $s\in S$ such that $\{x,y,s\}$ is a Schur triple;
    \item there is a loop at a vertex $x\in B$ if there exists $s,s'\in S$ such that $\{x,x,s\}$ or $\{x,s,s'\}$ is a Schur triple.
\end{enumerate}

As in \cite{ls} we distinguish two types of edges in the link graph. We call an edge $xy\in E(L_S[B])$ a \textit{type 1 edge} if $x-y=s$ for some $s\in S\cup(-S)$. Otherwise, if $x+y=s$ for some $s\in S$ we call it a \textit{type 2 edge}. Let $\Gamma:=L_S[B]$; we denote by $\Gamma_1$ and $\Gamma_2$ the subgraphs of $\Gamma$ consisting respectively of type 1 and type 2 edges. We write 
$d_i(x,\Gamma)$ for $d(x,\Gamma_i)$ (for $i=1,2$). Note that we might have some edges of both types. In this case they appear in both $\Gamma_1$ and $\Gamma_2$. 
Let $x\in B$ such that there is a loop at $x$ in $\Gamma$. If there are $s,s'\in S$ such that $\{x,s,s'\}$ is a Schur triple then we call this loop a \textit{bad loop}. Otherwise, if there is $s\in S$ such that $2x=s$ then we consider this loop as a type 2 edge, so it belongs to $\Gamma_2$.

The following lemma from \cite{BLST} provides a crucial connection between maximal sum-free sets and maximal independent sets in the link graph.

\begin{lemma}\cite{BLST}\label{extension link graph}
Suppose that $B,S\subseteq G$ are both sum-free. If $I\subseteq B$ is such that $S\cup I$ is a maximal sum-free subset of $G$, then $I$ is a maximal independent set in $L_S[B]$.
\end{lemma}

Note Lemma \ref{extension link graph} was proven in \cite{BLST} for sum-free subsets of $[n]$ but the proof for abelian groups is identical.

We now introduce some notation that will help generalise a lower bound construction for a group $H$ to a cartesian product $H\times K$.

\begin{definition}
Let $\Gamma_1=(V,E_1)$ and $\Gamma_2=(V,E_2)$ be two graphs with the same vertex set. We define the graph $\Gamma_1\rtimes \Gamma_2:=(V\times\{0,1\},E)$ as follows:
$$E:=\{((x,i),(y,i)) \ | \ (x,y)\in E_1, \ i=0,1\}\cup \{((x,0),(y,1)) \ | \ (x,y)\in E_2\}.$$
\end{definition}

Therefore, intuitively $\Gamma_1\rtimes \Gamma_2$ consists of two copies of $\Gamma_1$ connected by the edges of $\Gamma_2$. The following lemma shows how to build a link graph for a group $H\times K$ from a link graph for $H$. Recall that for a link graph $\Gamma:=L_S[B]$, $\Gamma_i$ denotes the subgraph of $\Gamma$ consisting of type $i$ edges, for $i=1,2$. We also write $\Gamma'$ for the graph obtained from $\Gamma$ by removing bad loops. Note that $\Gamma_i'=\Gamma_i$ for $i=1,2$.

\begin{lemma}\label{generalisation link graph}
Let $H,K$ be abelian groups and $a:=|\{k\in K\ | \ k+k=0_K\}|$. Let $B,S\subseteq H$ be disjoint sum-free sets and $\Gamma:=L_S[B]$. Then $\Tilde{B}:=B\times K$ and $\Tilde{S}:=S\times\{0_K\}$ are disjoint sum-free sets in $G:=H\times K$ and the graph $\Tilde{\Gamma}:=L_{\Tilde{S}}[\Tilde{B}]$ consists precisely of the following vertex-disjoint subgraphs:
one copy of $\Gamma$, $a-1$  copies of $\Gamma'$ and $(|K|-a)/2$ copies of $\Gamma_1\rtimes\Gamma_2$. In particular,
$$\mis(\Tilde{\Gamma})=\mis(\Gamma)\cdot\mis(\Gamma')^{a-1}\cdot\mis(\Gamma_1\rtimes\Gamma_2)^{(|K|-a)/2}.$$
\end{lemma}

\begin{proof}
By definition, $\Tilde{B}$ and $\Tilde{S}$ are disjoint sum-free sets. It is easy to check that $(b_1,k_1), (b_2,k_2)\in \Tilde{B}$ form a type 1 edge in $\Tilde{\Gamma}$ if and only if $b_1$ and $b_2$ form a type 1 edge in $\Gamma$ and $k_1=k_2$. Similarly, $(b_1,k_1), (b_2,k_2)\in \Tilde{B}$ form a type 2 edge in $\Tilde{\Gamma}$ if and only if $b_1$ and $b_2$ form a type 2 edge in $\Gamma$ and $k_1=-k_2$. Then $E_1(\Tilde{\Gamma})=\{((x,k),(y,k)) \ | \ (x,y)\in E_1(\Gamma),\ k\in K\}$ and $E_2(\Tilde{\Gamma})=\{((x,k),(y,-k)) \ | \ (x,y)\in E_2(\Gamma),\ k\in K\}$. Let $k\in K\backslash\{0_K\}$.
\begin{itemize}
    \item If $k=-k$, then the subgraph of $\Tilde{\Gamma}$ induced by the vertices of $B\times\{k\}$ is a copy of $\Gamma'$. Indeed, since $k\ne 0_K$ we cannot have any bad loop in the component induced by $B\times\{k\}$.
    \item If $k\ne-k$, then the subgraph of $\Tilde{\Gamma}$ induced by the vertices of $B\times\{k,-k\}$ is a copy of $\Gamma_1\rtimes\Gamma_2$. Indeed, an edge of this graph is either of the form $((b_1,k),(b_2,k))$ or $((b_1,-k),(b_2,-k))$ with $(b_1,b_2)\in E_1(\Gamma)$, or of the form $((b_1,k),(b_2,-k))$ with $(b_1,b_2)\in E_2(\Gamma)$. Conversely, every edge of one of those two forms is an edge of this subgraph.
\end{itemize}

If $k=0_K$ we have the same situation as in the first point, but with the bad loops, so the subgraph of $\Tilde{\Gamma}$ induced by the vertices of $B\times \{0_K\}$ is a copy of $\Gamma$. There are $a-1$ elements in $K\backslash\{0\}$ such that $k=-k$ and $(|K|-a)/2$ pairs $\{k,-k\}$ with $k\ne-k$. This now immediately yields the desired bound on $\mis(\Tilde{\Gamma})$.
\end{proof}

\section{Proof of Theorem~\ref{main}}\label{sec:mainthm}
\subsection{Upper bound}
In this section we prove the upper bound in Theorem \ref{main}. For this we will apply Theorem \ref{structure Z2} together with ideas from \cite{BLST2,ls} to show that the number of maximal sum-free sets from a specific class is negligible, and then we will count precisely the remaining maximal sum-free sets.

We apply Lemma \ref{containers} to $G=\Z_2^k$ and we consider any container $F$. By Lemma~\ref{containers}~(i) and a group removal lemma of Green \cite[Theorem 1.4]{G-R}, $F=B\cup C$ with $B$ sum-free and $|C|=o(n)$.
(For example, here we can choose $C$ to be a smallest subset of $F$ so that $B:=F\setminus C$ is sum-free.)

Then, every maximal sum-free set of $G$ in $F$ can be built in the following way:
\begin{enumerate}[label=(\arabic*), itemsep=0pt]
    \item Choose a (perhaps empty) sum-free set $S$ in $C$;
    \item Extend $S$ in $B$ to a maximal one.
\end{enumerate}
This two-step approach to constructing maximal sum-free sets will be crucial to our argument. With this in mind,
we now distinguish four types of maximal sum-free sets.
\begin{enumerate}
\item[\textbf{Type 0:}] those obtained from a container $F$ where $S$ is chosen to be empty.
    \item[\textbf{Type 1:}] those obtained from a container $F$ where $|B|\leq 5\cdot 2^{k-4}$ and $|S|\geq 1$.
    \item[\textbf{Type 2:}] those obtained from a container $F$ where $|B|> 5\cdot 2^{k-4}$ and  $|S|\geq 2$.
    \item[\textbf{Type 3:}] those obtained from a container $F$ where $|B|> 5\cdot 2^{k-4}$ and an $S$ which is a singleton.
\end{enumerate}

Let $f_{\max}^i(G)$ be the total number of maximal sum-free sets in $G$ of type $i$. Similarly we write $f_{\max}^i(F)$ for the number of maximal sum-free sets of type $i$ that lie in the container $F$. We will show that $f_{\max}^0(G)+f_{\max}^1(G)+f_{\max}^2(G)$ is negligible and then we will quantify  $f_{\max}^3(G)$ precisely. Note that each container $F$ can produce at most one type 0 maximal sum-free set (namely, $B$), so $f_{\max}^0(G)\leq 2^{n(\log n)^{-1/18}}$ by Lemma~\ref{containers}~(iii).

Fix any non-empty sum-free set $S$ in $C$. Note the zero element does not lie in $S$ as $S$ is sum-free. In what follows we count how many ways we can extend $S$ in $B$ to a maximal sum-free subset of $G$. Let $\Gamma:=L_S[B]$ be the link graph of $S$ on $B$.
From Lemma \ref{extension link graph} we have that the number of extensions of $S$ in $B$ to a maximal sum-free set is at most $\mis(\Gamma)$.

\begin{claim}\label{small B}
If $|B|\leq 5\cdot 2^{k-4}=5n/16$ then
$\mis(\Gamma)\leq 3^{5n/48}\leq  2^{0.17n}.$
\end{claim}

\begin{proof}
This follows immediately from (\ref{Moon-Moser}).
\end{proof}

We will now consider type 2 and type 3 maximal sum-free sets. Then we suppose $|B|>5\cdot 2^{k-4}$, so by Theorem \ref{structure Z2} we can assume that $B\subseteq x+U$ with $U$ a subspace of $G$ and $x\notin U$. In fact, (by adding elements to the container $F$ if necessary) we may assume $B=x+W$ where $W$ is a subspace with $\dim(W)=k-1$ and $x\notin W$, as $\fm^i(F)\leq \fm^i(F')$ for any $F'\supseteq F$, $i=2,3$. So $|B|=|G|/2=n/2$. Also, as $B$ and $S$ can be taken disjoint, we have $S\subseteq W$.

Under these conditions, we now show that the link graph is regular.

\begin{claim}\label{regularity of the link graph}
The link graph $\Gamma$ is $|S|$-regular.
\end{claim}

\begin{proof}
Let $b\in B$; then each $s\in S$ gives rise to exactly one neighbour of $b$ in $\Gamma$, namely $b+s$. Indeed, as $G=\Z_2^k$, if $b_1,b_2\in B$, $b_1+b_2=s$ if and only if $b_2=b_1+s$ if and only if $b_1=b_2+s$. Recall that since $S$ is sum-free, it does not contain the zero element. Then we cannot have loops in $\Gamma$ as $2b=0$ and $b+0=b$ are not valid solutions, and $b+s=s'$ means $b=s+s'\in W$ which is impossible.
\end{proof}

\begin{claim}\label{large S}
If $|S|>10^4$ then $\mis(\Gamma)\leq \left(\frac{n}{200}+1\right)\cdot2^{0.18n}$.
\end{claim}

\begin{proof}
We apply Lemma \ref{dense and regular graphs} with $k=1$ and $b\geq 100$. Therefore,
$$\mis(\Gamma)\leq\sum_{0\leq i\leq |B|/100}\binom{|B|}{i}\cdot 3^{\frac{|B|}{6}+\frac{2|B|}{300}}\leq \left(\frac{n}{200}+1\right)\cdot(100e)^{n/200}\cdot3^{13n/150}<\left(\frac{n}{200}+1\right)\cdot 2^{0.18n}.$$
\end{proof}

\begin{claim}\label{medium S}
If $4\leq |S|\leq 10^4$ then $\mis(\Gamma)\leq 3^{10^4/13}\cdot 2^{0.249n}$.
\end{claim}

\begin{proof}
Since $\Gamma$ is $|S|$-regular, $e(\Gamma)=\frac{|B||S|}{2}=|B|+\frac{|B|(|S|-2)}{2}$. So by Lemma~\ref{stability}, setting $C:=3^{10^4/13}$, $\Delta:=|S|$ and $k:=\frac{|B|(|S|-2)}{2}$, we have,

$$\mis(\Gamma)\leq C\cdot3^{\frac{|B|}{3}-\frac{|B|(|S|-2)}{26|S|}}\leq C\cdot 3^{49n/312}<C\cdot 2^{0.249n}.$$
\end{proof}

\begin{claim}\label{2S}
If $|S|=2$, $\mis(\Gamma)= 2^{n/8}$.
\end{claim}

\begin{proof}
When $s_1\ne s_2$, $L_{\{s_1,s_2\}}[B]$ is the disjoint union of $n/8$ cycles $C_4$ of the form $\{b,b+s_1,b+s_1+s_2,b+s_2\}$. Since $\mis(C_4)=2$, $\mis(L_{\{s_1,s_2\}}[B])=2^{n/8}$.
\end{proof}

\begin{claim}\label{3S}
If $|S|=3$, $\mis(\Gamma)= 6^{n/16}<2^{0.17n}$.
\end{claim}

\begin{proof}
Similarly to the previous claim, when $\{s_1,s_2,s_3\}$ is sum-free, $L_{\{s_1,s_2,s_3\}}[B]$ consists this time of $n/16$ disjoint copies of the cube $K_2\square K_2\square K_2$, which has 6 maximal independent sets. Indeed, as $\{s_1,s_2,s_3\}$ is sum-free, we can consider it as a family of linearly independent vectors over $\Z_2$, then the component of $\Gamma$ containing $b\in B$ is as in Figure 1. Hence $\mis(\Gamma)= 6^{n/16}$.
\begin{center}
\begin{tikzpicture}[scale=3.5]
\draw (0,0,0) -- ++(0,1,0) -- ++(1,0,0) -- ++(0,-1,0) -- cycle;
\draw (0,0,1) -- ++(0,1,0) -- ++(1,0,0) -- ++(0,-1,0) -- cycle;
\draw (0,0,1) -- ++(0,0,-1);
\draw (0,1,1) -- ++(0,0,-1);
\draw (1,0,1) -- ++(0,0,-1);
\draw (1,1,1) -- ++(0,0,-1);
\node at (-0.1,0,1) {$b$};
\node at (-0.2,1,1) {$b+s_3$};
\node at (1.25,0,1) {$b+s_1$};
\node at (1.45,1,0) {$b+s_1+s_2+s_3$};
\node at (1.32,0,0) {$b+s_1+s_2$};
\node at (-0.2,0,0) {$b+s_2$};
\node at (-0.32,1,0) {$b+s_2+s_3$};
\node at (1.35,1,1) {$b+s_1+s_3$};
\fill[black] (0,0,0) circle (1pt);
\fill[black] (0,0,1) circle (1pt);
\fill[black] (0,1,0) circle (1pt);
\fill[black] (0,1,1) circle (1pt);
\fill[black] (1,0,0) circle (1pt);
\fill[black] (1,0,1) circle (1pt);
\fill[black] (1,1,0) circle (1pt);
\fill[black] (1,1,1) circle (1pt);
\end{tikzpicture}

\textbf{Figure 1 - } The cube $K_2\square K_2\square K_2$.
\end{center}
\end{proof}

\begin{claim}\label{type 1 and 2 msfs}
$f_{\max}^0(G)+f_{\max}^1(G)+f_{\max}^2(G)\leq o(1)\cdot2^{n/4}$.
\end{claim}

\begin{proof}
We saw already that $f_{\max}^0(G)\leq 2^{n(\log n)^{-1/18}}$. Now we count type 2 maximal sum-free sets.
By Lemma \ref{containers} there are at most $2^{n(\log n)^{-1/18}}$ containers $F$. For each choice of $F$ there are $2^{o(n)}$ choices for $S$. By Claims~\ref{large S}--\ref{3S}, there are two constants $C,\alpha>0$ such that for each fixed $S$ with $|S|\geq 2$, $\mis(\Gamma)\leq C\cdot2^{(1/4-\alpha)n}$. Finally, by Lemma \ref{extension link graph},
$$f_{\max}^2(G)\leq C\cdot2^{n(\log n)^{-1/18}}\cdot 2^{o(n)}\cdot 2^{(1/4-\alpha)n}.$$
Therefore, for $n$ sufficiently large, $f_{\max}^2(G)\leq C\cdot 2^{(1/4-\alpha/2)n}=o(1)\cdot2^{n/4}$.
The analogous argument applied with Claim~\ref{small B} implies that $f_{\max}^1(G)\leq o(1)\cdot2^{n/4}$.
\end{proof}

\begin{claim}\label{type 3 msfs}
$f_{\max}^3(G)\leq \frac{(n-2)(n-1)}{2}\cdot 2^{n/4}.$
\end{claim}

\begin{proof}
To give an upper bound on $\fm^3(G)$ recall it suffices to assume each relevant container is of the form $F=B\cup C$ where $B=x+W$ and $W$ is a subspace of $G$ with $\dim(W)=k-1$; $x\notin W$; $|C|=o(n)$. Moreover, for type 3 maximal sum-free sets, $S$ is a singleton. Thus, to upper bound $\fm^3(G)$ it suffices to sum up the number of maximal independent sets in all link graphs of the form $L_{\{s\}}[x+W]$ where $\dim(W)=k-1$; $x\notin W$ and $\{s\}\subseteq W$ is a sum-free singleton, so any singleton except $\{0\}$.

Let $W\subseteq G$ be any $(k-1)$-dimensional subspace of $G$, $x\notin W$ and $s\in W\backslash\{0\}$. Let $\Gamma:=L_{\{s\}}[x+W]$ be the link graph of $\{s\}$ on $x+W$. By Claim \ref{regularity of the link graph}, $\Gamma$ is $1$-regular so it is a perfect matching between the elements of $x+W$. Since $|x+W|=n/2$, $\mis(\Gamma)=2^{|x+W|/2}=2^{n/4}$.

Now we have to count how many such pairs $(x+W,\{s\})$ we have. There are $n-1$ subspaces $W$ in $G$ of dimension $k-1$, and exactly one coset for each, thus there are $n-1$ choices for $x+W$. Once $x+W$ is fixed, since $\{s\}$ is a sum-free singleton in $W$, we have $|W\backslash\{0\}|=n/2-1$ choices for $\{s\}$. It may be that we count the same maximal sum-free set for two different pairs $(x+W,\{s\})$, but here we are just looking for an upper bound so we do not mind over-counting yet. Overall, we have $(n-1)(n/2-1)=(n-1)(n-2)/2$ such pairs and $2^{n/4}$ maximal sum-free sets generated by each pair, thus $f_{\max}^3(G)\leq \frac{(n-2)(n-1)}{2}\cdot 2^{n/4}$.
\end{proof}
Since every maximal sum-free subset of $G$ lies in a container,  Claims \ref{type 1 and 2 msfs} and \ref{type 3 msfs} give us the desired upper bound: 
$$f_{\max}(G)\leq f_{\max}^0(G)+ f_{\max}^1(G)+f_{\max}^2(G)+f_{\max}^3(G)\leq \left(\binom{n-1}{2}+o(1)\right) 2^{n/4}.$$

\subsection{Lower bound}\label{sec:lb Z2}

In this subsection we show that type 3 maximal sum-free sets are those that contribute to the dominant term in the asymptotic development of $f_{\max}(G)$. For a coset $B=x+W$ with $W\subseteq G$ a $(k-1)$-dimensional subspace; $x\notin W$ and $S\subseteq W$ sum-free, we say that the pair $(B,S)$ \textit{generates} a sum-free set $A$ if there is a maximal independent set $I$ in the link graph $L_S[B]$ such that $A=I\cup S$.

\begin{claim}\label{dominant term Z2}
Consider any coset $B=x+W$ with $W\subseteq G$ a $(k-1)$-dimensional subspace; $x\notin W$ and any singleton $\{s\}\subseteq W\backslash\{0\}$. Then  $(B,\{s\})$ generates at least $2^{n/4}-(n/2-2)\cdot2^{n/8}$ maximal sum-free sets.
\end{claim}

\begin{proof}
We fix $B=x+W$ with $\dim(W)=k-1$ and $x\notin W$, and $s\in W\backslash\{0\}$. We saw that $L_{\{s\}}[B]$ has $2^{n/4}$ maximal independent sets. Let $I\subseteq B$ be a maximal independent set in $L_{\{s\}}[B]$ and suppose that $\{s\}\cup I$ is not a maximal sum-free set in $G$; call such an $I$ \emph{bad}. Then there exists $s'\in W\backslash\{0,s\}$ such that $\{s,s'\}\cup I$ is sum-free ($s'$ cannot belong to $B$ otherwise $I$ would not be a maximal independent set). So $I$ is a maximal independent set in $L_{\{s,s'\}}[B]$.
By Claim~\ref{2S}, there are $2^{n/8}$ maximal independent sets in $L_{\{s,s'\}}[B]$. Since there are $n/2-2$ possibilities for $s'$, in total there are at most $(n/2-2)\cdot2^{n/8}$ bad $I$.
The claim immediately follows.


\end{proof}

Now we have to be careful because we may count the same maximal sum-free set in two different pairs $(B,\{s\})$. 
\begin{claim}\label{overcounting Z2}
Given distinct pairs $(B,\{s\})$, $(B',\{s'\})$ there are at most $n/4$ maximal sum-free subsets of $G$ that are generated by both pairs.
\end{claim}
\begin{proof}
Let $(B,\{s\})=(x+W,\{s\})$ and $(B',\{s'\})=(x'+W',\{s'\})$ be two distinct such pairs.

\begin{itemize}
    \item If $B=B'$ then necessarily $s\ne s'$ and so maximal sum-free sets generated by these pairs differ by at least one element.
    \item If $B\ne B'$ then $|B\cap B'|\leq n/4$. Indeed, as $B$ and $B'$ are cosets, $B \cap B'$ is either empty or is a coset of a subspace of co-dimension two.
    
    Note the maximal sum-free sets generated by $(B,\{s\})$ (resp. $(B',\{s'\})$)  consist of $n/4$ elements of $B$ (resp. $B'$) plus $s$ (resp. $s'$).
    \begin{itemize}[label=-,itemsep=0pt]
        \item If $s=s'$ then a maximal sum-free set generated by both pairs contains necessarily the whole intersection $B\cap B'$ and $s=s'$, so there is only one such set.
        \item If $s\ne s'$ we can assume $s\in B'$ and $s'\in B$ otherwise a maximal sum-free set cannot be generated by both pairs for the same reason as the first point. To build a maximal sum-free set generated by both pairs, we need to choose obviously $s$ and $s'$, and $n/4-1$ elements in $B\cap B'$. There are at most $n/4$ ways to do so, hence there are at most $n/4$ maximal sum-free sets generated by both pairs.
    \end{itemize}
\end{itemize}
Altogether this shows that the number of maximal sum-free sets that are generated by both $(B,\{s\})$ and $(B',\{s'\})$ is at most $n/4$. 
\end{proof}

As there are $\binom{(n-2)(n-1)/2}{2}$ couples of such pairs, we over-count only a number of maximal sum-free sets which is polynomial in $n$, and so $o(1)\cdot 2^{n/4}$. Altogether we have that
\begin{align*}
    f_{\max}(G)&\geq \frac{(n-2)(n-1)}{2}(2^{n/4}-(n/2-2)\cdot2^{n/8})-o(1)\cdot2^{n/4}\\
    &=\left(\binom{n-1}{2}-o(1)\right)2^{n/4}.
\end{align*}

\section{Proof of Theorem~\ref{thm2}}\label{sec:thm2}

\subsection{Upper bound}
We will use the same method as in Section \ref{sec:mainthm}, though instead of Theorem~\ref{structure Z2} we will apply
Theorem \ref{structure Z3}. 

We apply Lemma \ref{containers} to $G=\Z_3^k$. Consider any container $F$; by Green's removal lemma \cite[Theorem 1.5]{G-R} we have that $F=B\cup C$ with $B$ sum-free and $|C|=o(n)$. 
As in Section~\ref{sec:mainthm},
every maximal sum-free subset of $G$ in $F$ can be built in the following way:
\begin{enumerate}[label=(\arabic*), itemsep=0pt]
    \item Choose a (perhaps empty) sum-free set $S$ in $C$;
    \item Extend $S$ in $B$ to a maximal one.
\end{enumerate}
We then define the following types of maximal sum-free sets.
\begin{enumerate}
\item[\textbf{Type 0:}] those obtained from a container $F$ where $S$ is chosen to be empty.
    \item[\textbf{Type 1:}] those obtained from a container $F$ where $|B|\leq 5\cdot 3^{k-3}$ and $|S|\geq 1$.
    \item[\textbf{Type 2:}] those obtained from a container $F$ where $|B|> 5\cdot 3^{k-3}$ and  $|S|\geq 2$.
    \item[\textbf{Type 3:}] those obtained from a container $F$ where $|B|> 5\cdot 3^{k-3}$ and an $S$ which is a singleton.
\end{enumerate}

We define $\fm^i(G)$ and $\fm^i(F)$, $i=0,1,2,3$, analogously to before. As before, $f_{\max}^0(G)\leq 2^{n(\log n)^{-1/18}}$.

Fix any non-empty sum-free $S\subseteq C$ and let $\Gamma:=L_S[B]$ be the link graph of $S$ on $B$.

\begin{claim}\label{small B in Z3}
If $|B|\leq 5\cdot 3^{k-3}=5n/27$ then
$\mis(\Gamma)\leq 3^{5n/81}<3^{0.07n}.$
\end{claim}

\begin{proof}
This follows immediately from (\ref{Moon-Moser}).
\end{proof}

We will now consider type 2 and type 3 maximal sum-free sets. Thus, we suppose $|B|>5\cdot 3^{k-3}$, so by Theorem \ref{structure Z3}, $B\subseteq g+H$ with $H$ a subspace of $G$ over $\Z_3$ and $g\notin H$. Analogously to before we may assume that $B=g+H$ with $\dim(H)=k-1$. So $|B|=|G|/3=n/3$. Moreover, as $B$ and $S$ can be taken disjoint, we have $S\subseteq H\cup(2g+H)$.

The following claim describes the regularity of the link graph.

\begin{claim}\label{regularity of the link graph Z3}
For all $x\in B$,
\begin{itemize}
    \item $d_1(x,\Gamma)=|(S\cup(-S))\cap H|$;
    \item $|(2g+H)\cap S|\leq d_2(x,\Gamma)\leq |(2g+H)\cap S|+1$.
\end{itemize}
Furthermore,
\begin{equation}\label{degrees link graph}
|S|\leq |(S\cup(-S))\cap H|+|(2g+H)\cap S|\leq \delta(\Gamma)\leq \Delta(\Gamma)\leq 2|S|+1\leq 3\delta(\Gamma).
\end{equation}
\end{claim}

\begin{proof}
For the first point, each element $s\in (S\cup(-S))\cap H$ gives rise to exactly one type 1 neighbour of $x$ in $\Gamma$, namely $x+s$. Conversely, if there are $x,y\in B$ and $s\in S$ such that $s=x-y$ then necessarily $s\in H$. For the second point, note that each element $s\in (2g+H)\cap S$ creates one loop at $x\in g+H$ such that $2x=s$ (this $x$ is unique). So for a given $x\in B$, either (i) $2x\notin (2g+H)\cap S$ and so each $s\in (2g+H)\cap S$ gives one neighbour to $x$ which is different from $x$
or (ii) $2x\in (2g+H)\cap S$ and we have the same situation but one $s$ creates a loop at $x$ which contributes two to the degree of $x$.
Conversely, if there are $x,y\in B$ and $s\in S$ such that $x+y=s$ then necessarily $s\in 2g+H$. Thus, the second point holds.

Our argument above therefore implies that $|(S\cup(-S))\cap H|+|(2g+H)\cap S|\leq \delta(\Gamma)$.
Note that if $(2g+H)\cap S=\emptyset$ then there cannot be any bad loops in $\Gamma$. Therefore, in this case we have that $\Delta(\Gamma)=|(S\cup(-S))\cap H|\leq 2|S| $. 

So assume that $(2g+H)\cap S\ne \emptyset$. Given any $s \in (2g+H)\cap S$ notice that $-s \not \in -S\cap H$. This implies that
$|(S\cup(-S))\cap H|+|(2g+H)\cap S| \leq 2(|S|-|(2g+H)\cap S| )+|(2g+H)\cap S| =2|S|-|(2g+H)\cap S| \leq 2|S|-1$.
Given any $x\in B$, if $d_2(x,\Gamma)= |(2g+H)\cap S|$ then we may have a bad loop at $x$ in $\Gamma$ and so
$$d(x,\Gamma)\leq |(S\cup(-S))\cap H|+|(2g+H)\cap S|+2\leq 2|S|+1.$$
Otherwise, if  $d_2(x,\Gamma)= |(2g+H)\cap S|+1$ then the loop at $x$ in $\Gamma$ is already counted here and so 
$$d(x,\Gamma)\leq |(S\cup(-S))\cap H|+|(2g+H)\cap S|+1\leq 2|S|.$$
In all cases we have shown that $\Delta(\Gamma)\leq 2|S|+1$.

The first inequality holds in (\ref{degrees link graph}) since $$|S|=|S\cap H|+|S\cap(2g+H)|\leq |(S\cup(-S))\cap H|+|(2g+H)\cap S|.$$ Altogether this shows that all inequalities in (\ref{degrees link graph}) hold.
\end{proof}

\begin{claim}\label{large S in Z3}
If $|S|>10^4$ then $\mis(\Gamma)\leq \left(\frac{n}{300}+1\right)\cdot3^{0.103n}$.
\end{claim}

\begin{proof}
By (\ref{degrees link graph}) we can apply Lemma \ref{dense and regular graphs} with $k=3$ and $b\geq 100$:
$$\mis(\Gamma)\leq\sum_{0\leq i\leq n/300}\binom{n/3}{i}\cdot 3^{\frac{n}{12}+\frac{2n}{900}}\leq \left(\frac{n}{300}+1\right)\cdot(100e)^{n/300}\cdot3^{77n/900}<\left(\frac{n}{300}+1\right)\cdot 3^{0.103n}.$$
\end{proof}

\begin{claim}\label{medium S in Z3}
If $3\leq |S|\leq 10^4$ then $\mis(\Gamma)\leq 3^{3\cdot10^4/13}\cdot 3^{(1/9-1/702)n}$.
\end{claim}

\begin{proof}
By Claim \ref{regularity of the link graph Z3}, $\Delta(\Gamma)\leq 3|S|\leq 3\cdot 10^4$ and $$
e(\Gamma)\geq\frac{|B|\delta(\Gamma)}{2}
\stackrel{(\ref{degrees link graph})}{
\geq} \frac{|B||S|}{2}.
$$
Then, by Lemma \ref{stability}, with $C=3^{3\cdot10^4/13}$ and $k=\frac{|B|(|S|-2)}{2}$,
\begin{align*}
    \mis(\Gamma)&\leq 3^{3\cdot10^4/13}\cdot 3^{\frac{|B|}{3}\left(1-\frac{(|S|-2)}{26\cdot|S|}\right)}
    \leq 3^{3\cdot10^4/13}\cdot 3^{\frac{n}{9}-\frac{n(|S|-2)}{26\cdot9|S|}}
    \leq 3^{3\cdot10^4/13}\cdot 3^{(1/9-1/702)n}.
\end{align*}
\end{proof}

\begin{claim}\label{2S Z3}
If $|S|=2$, $\mis(\Gamma)\leq 3^{0.1n}$.
\end{claim}

\begin{proof}
There are three cases depending on which coset the elements of $S$ belong to. Note that since $S$ is sum-free, we cannot have $S=\{s,-s\}=\{s,2s\}$ for an $s\in H\cup (2g+H)$, and the zero element does not belong to $S$.
\begin{enumerate}[label=\textbf{Case \arabic*:},leftmargin=*,itemsep=5pt]
    \item $S=\{s_1,s_2\}$ with $s_1,s_2\in 2g+H$. Let $x_1:=2s_1\in B$ and $x_2:=2s_2\in B$ be the unique elements such that $2x_1=s_1$ and $2x_2=s_2$. Note that $x_2=s_1-s_2+x_1$. The component of $\Gamma$ containing $x_1$ and $x_2$ is the path $\{x_1,s_2-x_1=s_1-x_2=s_1+s_2,x_2\}$ with a bad loop at each of those three vertices (the loops at $x_1$ and $x_2$ are both bad loops and type 2 edges). Note the empty set is the only maximal independent set of this component. The component in $\Gamma$ of any $x\in B\backslash\{x_1,s_1+s_2,x_2\}$ is the cycle $C_6$ as shown in Figure~2. Then $\Gamma$ consists of the component of $x_1$ and $x_2$ and $(|B|-3)/6$ disjoint copies of the cycle $C_6$. Since $\mis(C_6)=5$, $\mis(\Gamma)=5^{(|B|-3)/6}=\frac{5^{n/18}}{\sqrt{5}}<3^{0.1n}$.
    
\begin{center}
\begin{tikzpicture}[scale=2]
\draw (0,0) -- ++(-0.5,0.866) -- ++(-1,0) -- ++(-0.5,-0.866) -- ++(0.5,-0.866) -- ++(1,0) -- cycle;
\node at (0.65,0) {$-s_1-s_2-x$};
\node at (-2.2,0) {$x$};
\node at (-1.9,0.866) {$s_1-x$};
\node at (0.15,0.866) {$-s_1+s_2+x$};
\node at (0.15,-0.866) {$s_1-s_2+x$};
\node at (-1.9,-0.866) {$s_2-x$};
\fill[black] (0,0) circle (1.5pt);
\fill[black] (-0.5,0.866) circle (1.5pt);
\fill[black] (-1.5,0.866) circle (1.5pt);
\fill[black] (-2,0) circle (1.5pt);
\fill[black] (-1.5,-0.866) circle (1.5pt);
\fill[black] (-0.5,-0.866) circle (1.5pt);
\end{tikzpicture}

\textbf{Figure 2 - } The 6-cycle $C_6$.
\end{center}

    \item $S=\{s_1,s_2\}$ with $s_1\in H$ and $s_2\in 2g+H$. Let $y:=2s_2\in B$ be the unique element such that $2y=s_2$. Then the component of $y$ is the triangle $\{y,s_1+y=s_1-s_2,-s_1+y=-s_1-s_2\}$ with a bad loop at $y$ (which is also a type 2 edge) and a bad loop at $s_1-s_2$; so this component has one maximal independent set. The component in $\Gamma$ of any $x\in B\backslash\{y,s_1-s_2,-s_1-s_2\}$ is the graph $K_2\square K_3$ as represented in Figure~3. In particular, if $x$ is not a vertex in the component of $y$, then one can check that all vertices in Figure~3 are distinct.
    
    Then $\Gamma$ consists of the component of $y$ and $(|B|-3)/6$ disjoint copies of $K_2\square K_3$. Since $\mis(K_2\square K_3)=6$, $\mis(\Gamma)=6^{(|B|-3)/6}=\frac{6^{n/18}}{\sqrt{6}}<3^{0.1n}$.
    
\begin{center}
\begin{tikzpicture}[scale=1.5]
\draw (-0.5,0) -- ++(-1.5,0.866) -- ++(0,-1.732) -- cycle;
\draw (0.5,0) -- ++(1.5,0.866) -- ++(0,-1.732) -- cycle;
\draw (-0.5,0) -- (0.5,0);
\draw (-2,0.866) -- (2,0.866);
\draw (-2,-0.866) -- (2,-0.866);
\node at (-0.5,0.2) {$x$};
\node at (0.4,0.2) {$s_2-x$};
\node at (-2.5,0.866) {$s_1+x$};
\node at (-2.6,-0.866) {$-s_1+x$};
\node at (2.85,0.866) {$-s_1+s_2-x$};
\node at (2.75,-0.866) {$s_1+s_2-x$};
\fill[black] (0.5,0) circle (2pt);
\fill[black] (-0.5,0) circle (2pt);
\fill[black] (-2,0.866) circle (2pt);
\fill[black] (-2,-0.866) circle (2pt);
\fill[black] (2,0.866) circle (2pt);
\fill[black] (2,-0.866) circle (2pt);
\end{tikzpicture}

\textbf{Figure 3 - } The graph $K_2\square K_3$.
\end{center}

    \item $S=\{s_1,s_2\}$ with $s_1,s_2\in H$, $s_1\ne s_2$ and $s_1\ne -s_2$. In this case there is no loop in $\Gamma$. The component of any $x\in B$ in $\Gamma$ is represented in Figure~4. Then $\Gamma$ consists of $|B|/9$ disjoint copies of this graph. Since it has 6 maximal independent sets, $\mis(\Gamma)=6^{|B|/9}=6^{n/27}<3^{0.1n}$.
    
\begin{center}
\begin{footnotesize}
\begin{tikzpicture}[scale=2]
\draw (0,0) -- ++(1,0) -- ++(0,1) -- ++(-2,0) -- ++(0,-2) -- ++(2,0) -- ++(0,1);
\draw (0,0) -- (0,1);
\draw (0,0) -- (0,-1);
\draw (0,0) -- (-1,0);
\fill[black] (0,0) circle (1.5pt);
\fill[black] (0,1) circle (1.5pt);
\fill[black] (-1,0) circle (1.5pt);
\fill[black] (1,0) circle (1.5pt);
\fill[black] (0,-1) circle (1.5pt);
\fill[black] (-1,-1) circle (1.5pt);
\fill[black] (1,1) circle (1.5pt);
\fill[black] (-1,1) circle (1.5pt);
\fill[black] (1,-1) circle (1.5pt);
\draw (1,-1) arc (60:120:2cm);
\draw (1,0) arc (60:120:2cm);
\draw (1,1) arc (60:120:2cm);
\draw (-1,-1) arc (-120:-240:1.15cm);
\draw (0,-1) arc (-120:-240:1.15cm);
\draw (1,-1) arc (-120:-240:1.15cm);
\node at (-0.13,0.13) {$x$};
\node at (0,-1.15) {$x-s_2$};
\node at (0,1.15) {$x+s_2$};
\node at (-1.3,0) {$x-s_1$};
\node at (1.33,0) {$x+s_1$};
\node at (-1.4,1.15) {$x-s_1+s_2$};
\node at (1.4,1.15) {$x+s_1+s_2$};
\node at (-1.4,-1.15) {$x-s_1-s_2$};
\node at (1.4,-1.15) {$x+s_1-s_2$};
\end{tikzpicture}
\end{footnotesize}

\textbf{Figure 4 - } The network $\Z_3^2$.
\end{center}

\end{enumerate}
\end{proof}

\begin{claim}\label{type 1 and 2 msfs Z3}
$\fm^0(G)+f_{\max}^1(G)+f_{\max}^2(G)\leq o(1)\cdot3^{n/9}$.
\end{claim}

\begin{proof}
This follows analogously to the proof of Claim \ref{type 1 and 2 msfs} where now we apply Claims~\ref{small B in Z3}, \ref{large S in Z3}, \ref{medium S in Z3} and \ref{2S Z3}
together with Lemma~\ref{containers}.
\end{proof}

\begin{claim}\label{type 3 msfs in Z3}
$f_{\max}^3(G)\leq \left(\frac{(n-3)(n-1)}{3}+o(1)\right)\cdot 3^{n/9}$
\end{claim}

\begin{proof}
Similarly to the proof of Claim \ref{type 3 msfs}, it suffices to sum up the number of maximal independent sets in link graphs of the form $L_{\{s\}}[g+H]$ where $\dim(H)=k-1$, $g\notin H$ and $\{s\}\subseteq H\cup (2g+H)$ is a sum-free singleton (so any singleton except $\{0\}$).

Let $\{s\}\subseteq H\cup (2g+H)$ be such a singleton and $B=g+H$. We will distinguish two cases depending on which set $s$ belongs to.
\begin{itemize}
    \item If $s\in 2g+H$ then there is one loop in $\Gamma$ at $y\in B$ such that $2y=s$ and $\Gamma$ consists of a perfect matching between the $(|B|-1)/2$ pairs $\{x,s-x\}$ with $x\ne y$. Then $\mis(\Gamma)=2^{(|B|-1)/2}=\frac{2^{n/6}}{\sqrt{2}}$.
    \item If $s\in H$ then $\Gamma$ consists of $|B|/3$ disjoint triangles of the form $\{x-s,x,x+s\}$; thus $\mis(\Gamma)=3^{|B|/3}=3^{n/9}$.
\end{itemize}

Now we want to count precisely how many pairs $(g+H,\{s\})$ we have. Since there are $(n-1)/2$ $(k-1)$-dimensional subspaces of $G$, we have $(n-1)/2$ choices for $H$ and two choices for the coset. Once $g+H$ is fixed, we have $|H\backslash\{0\}|=n/3-1$ choices for $s\in H$ and $|2g+H|=n/3$ choices for $s\in 2g+H$. As we do not mind over-counting for now, we conclude that
$$f_{\max}^3(G)\leq (n-1)(n/3-1)\cdot 3^{n/9} + \frac{n(n-1)}{6\sqrt{2}}\cdot2^{n/6}=\left(\frac{(n-3)(n-1)}{3}+o(1)\right) 3^{n/9}.$$
\end{proof}

Combining Claims \ref{type 1 and 2 msfs Z3} and \ref{type 3 msfs in Z3} gives us the upper bound 
$$f_{\max}(G)\leq \fm^0(G)+ f_{\max}^1(G)+f_{\max}^2(G)+f_{\max}^3(G)\leq \left(\frac{(n-3)(n-1)}{3}+o(1)\right)\cdot 3^{n/9}.$$

\subsection{Lower bound}
Given a coset $g+H$ and $s\in H\backslash\{0\}$ 
where $\dim(H)=k-1$,
we define when $(g+H,\{s\})$ \emph{generates} a maximal sum-free subset of $G$ analogously to Section~\ref{sec:lb Z2}.
Similar to the behaviour for $\mathbb Z^k_2$,  type 3 maximal sum-free sets generated by a pair $(g+H,\{s\})$ with $s\in H\backslash\{0\}$ are those that contribute to the dominant term in the asymptotic development of $f_{\max}(G)$.

\begin{claim}
Consider any coset $B=g+H$ with $H\subseteq G$ a $(k-1)$-dimensional subspace; $g\notin H$ and any singleton $\{s\}\subseteq H\backslash\{0\}$. Then $(B,\{s\})$ generates at least $3^{n/9}-(2n/3-2)\cdot 3^{0.1n}$ maximal sum-free sets.

\end{claim}

\begin{proof}
We fix $B=g+H$ with $\dim(H)=k-1$ and $g\notin H$, and $s\in H\backslash\{0\}$. We saw in the proof of Claim \ref{type 3 msfs in Z3} that $L_{\{s\}}[B]$ has $3^{n/9}$ maximal independent sets. Let $I\subseteq B$ be a maximal independent set in $L_{\{s\}}[B]$ and suppose that $\{s\}\cup I$ is not a maximal sum-free set in $G$; call such an $I$ \textit{bad}. Then there exists $s'\in H\backslash\{0,s\}\cup(2g+H)$ such that $\{s,s'\}\cup I$ is sum-free ($s'$ cannot belong to $B$ otherwise $I$ would not be a maximal independent set). So $I$ is a maximal independent set in $L_{\{s,s'\}}[B]$. By Claim \ref{2S Z3}, there are at most $3^{0.1n}$ maximal independent sets in $L_{\{s,s'\}}[B]$. Since there are $2n/3-2$ possibilities for $s'$, in total there are at most $(2n/3-2)\cdot 3^{0.1n}$ bad $I$. The claim immediately follows.
\end{proof}

As in Claim \ref{overcounting Z2}, there are only a few maximal sum-free sets generated by two distinct pairs $(g+H,\{s\})$.
\begin{claim}
Given distinct pairs $(B,\{s\}),(B',\{s'\})$ there are at most $n/9$ maximal sum-free subsets of $G$ generated by both pairs.
\end{claim}

\begin{proof}
Let $(B,\{s\})=(g+H,\{s\})$ and $(B',\{s'\})=(g'+H',\{s'\})$ be two distinct such pairs.

\begin{itemize}
    \item If $H=H'$ and $g+H=g'+H'$ then necessarily  $s\ne s'$ and so maximal sum-free sets generated by these pairs differ by at least one element. If $H=H'$ and $g+H=2g'+H'=2g'+H$ then $B\cap B'=\emptyset$ so maximal sum-free sets generated by these pairs are distinct.
    \item If $H\ne H'$ then $|B\cap B'|\leq n/9$. Indeed, 
    since $B$ and $B'$ are cosets, $B\cap B'$ is either empty or is a coset of a subspace of co-dimension two.
    
    Since the maximal sum-free sets generated by $(B,\{s\})$ (resp. $(B',\{s'\})$) consist of $n/9$ elements of $B$ (resp. $B'$) plus $s$ (resp. $s'$), we can conclude similarly to the second point in the proof of Claim \ref{overcounting Z2} that there are at most $n/9$ maximal sum-free sets generated by both pairs.
\end{itemize}
\end{proof}

As there are $\binom{(n-3)(n-1)/3}{2}$ couples of such pairs, we over-count only a number of maximal sum-free sets which is polynomial in $n$, and so $o(1)\cdot3^{n/9}$. Altogether we have that
\begin{align*}
    f_{\max}(G)&\geq \frac{(n-3)(n-1)}{3}(3^{n/9}-(2n/3-2)\cdot 3^{0.1n})-o(1)\cdot3^{n/9}\\
    &=\left(\frac{(n-3)(n-1)}{3}-o(1)\right)3^{n/9}.
\end{align*}

\section{Lower bound construcions}\label{sec:construct}

In this section we provide constructions to confirm the lower bound in Conjecture~\ref{conj1} for a range of groups.

\subsection{Groups with a large cyclic component} In \cite[Proposition 5.4]{BLST2} the authors give a construction which confirms Conjecture \ref{conj1} for the cyclic group $\Z_m$. We will extend it to groups that have a large cyclic component.

\begin{lemma}\label{large cyclic component}
Let $m\geq 9 $ be an integer. For any $n$-order abelian group $G$ of the form $G=\Z_m\times K$ we have
$$f_{\max}(G)\geq (2/3)^{1+n/m}\cdot 6^{(\frac{1}{18}-\frac{4}{9m})n}.$$
\end{lemma}

\begin{proof}
Consider the following construction for $\Z_m$. We write $m=9k+i$ with $0\leq i\leq 8$. We set $B:=[3k+1,6k]$ and $S:=\{k,-2k\}$. Note  both $B$ and $S$ are sum-free. 

Let us first assume that $i\geq 1$. The link graph $\Gamma:=L_S[B]$ can have four different shapes depending on the parities of $i-1$ and $k-i+1$. When $k-i+1$ is even, for any $x\in [3k+i,4k]$, $x\ne -x-2k$, so its component in $\Gamma$ is the graph $K_2 \square K_3$, as shown in Figure~5. When $k-i+1$ is odd, we have the same situation except when $x=3k+i+\frac{k-i}{2}$ whose component is a triangle with a type 2 loop at one vertex, since $2(3k+i+\frac{k-i}{2})=-2k$. We call this graph represented in Figure~7 a \textit{looped triangle}. Similarly, when $i-1$ is even, the component of any $x\in [3k+1,3k+i-1]$ is a $K_2\square K_3$ as shown in Figure~6. It is identical when $i-1$ is odd except for $3k+i/2$ whose component is a looped triangle as shown in Figure 8. Note also there is always a bad loop at $5k+i=-2k-2k$ which belongs to a $K_2\square K_3$ component.

We are now able to compute $\Gamma$, $\Gamma'$ and $\Gamma_1\rtimes\Gamma_2$ in each of those four cases. We do it here when $i-1$ and $k-i+1$ are both odd, the three other cases are similar. From the previous observations we deduce that $\Gamma$ consists of $k/2-2$ vertex-disjoint copies of $K_2\square K_3$, one copy of $K_2\square K_3$ with a bad loop at one vertex and two looped triangles. Then $\Gamma'$ is identical to $\Gamma$ but without the bad loop; from Figures 5-8 we deduce that $\Gamma_1\rtimes \Gamma_2$ consists of $k$ copies of $K_2\square K_3$. Indeed, when $\Lambda=K_2\square K_3$ as in Figures 5 and 6, $\Lambda_1\rtimes\Lambda_2$ consists of 2 vertex-disjoint copies of $K_2\square K_3$; when $\Lambda$ is a looped triangle as in Figures 7 and 8, $\Lambda_1\rtimes\Lambda_2$ is $K_2\square K_3$. 

Moreover $K_2\square K_3$ has 6 maximal independent sets, $K_2\square K_3$ with a loop at one vertex has 4 and a looped triangle has 2. Thus $\mis(\Gamma)=6^{k/2-2}\cdot4\cdot2^2$, $\mis(\Gamma')=6^{k/2-1}\cdot2^2$ and $\mis(\Gamma_1\rtimes\Gamma_2)=6^{k}$. 
Define $a$, $\Tilde{\Gamma}$, $\Tilde{S}$ and $\Tilde{B}$ as in Lemma~\ref{generalisation link graph}. Then by
Lemma~\ref{generalisation link graph},
\begin{align*}
    \mis(\Tilde{\Gamma})&=6^{k/2-2}\cdot4\cdot2^2\cdot6^{(k/2-1)(a-1)}\cdot2^{2(a-1)}\cdot6^{\frac{k}{2}(|K|-a)}\\
    &=(4/6)^{a+1}\cdot  6^{\frac{k}{2}|K|} \geq (2/3)^{|K|+1}\cdot6^{\frac{m-i}{18}|K|}\\
    &\geq (2/3)^{|K|+1}\cdot6^{\frac{m-8}{18}|K|}= (2/3)^{n/m+1}\cdot6^{(\frac{1}{18}-\frac{4}{9m})n}.
\end{align*}

A maximal independent  set $I$ in $\Tilde{\Gamma}$ is such that $I \cup \Tilde{S}$ is sum-free, but it might not be maximal. However no further element from $\Tilde{B}$ can be added. Hence two distinct maximal independent sets $I$, $I'$ in $\Tilde{\Gamma}$ are such that $I \cup \Tilde{S}$ and $I' \cup \Tilde{S}$ lie in two distinct maximal sum-free subsets of $G$. Therefore we have $\fm(G)\geq \mis(\Tilde{\Gamma})\geq (2/3)^{1+n/m}\cdot 6^{(\frac{1}{18}-\frac{4}{9m})n}$. We obtain greater or equal bounds in the other three cases.

\begin{multicols}{2}

\begin{tikzpicture}[scale=1.3]
\draw[blue] (-0.5,0) -- ++(-1.5,0.866) -- ++(0,-1.732) -- cycle;
\draw[blue] (0.5,0) -- ++(1.5,0.866) -- ++(0,-1.732) -- cycle;
\draw[red] (-0.5,0) -- (0.5,0);
\draw[red] (-2,0.866) -- (2,0.866);
\draw[red] (-2,-0.866) -- (2,-0.866);
\node at (-0.3,0.2) {$x+k$};
\node at (0.45,-0.2) {$-x$};
\node at (-2.25,0.866) {$x$};
\node at (-2.55,-0.866) {$x+2k$};
\node at (2.65,0.866) {$-x-2k$};
\node at (2.6,-0.866) {$-x-k$};
\fill[black] (0.5,0) circle (2pt);
\fill[black] (-0.5,0) circle (2pt);
\fill[black] (-2,0.866) circle (2pt);
\fill[black] (-2,-0.866) circle (2pt);
\fill[black] (2,0.866) circle (2pt);
\fill[black] (2,-0.866) circle (2pt);
\end{tikzpicture}

\begin{center}
\textbf{Figure 5 - } The component of $x\in [3k+i,4k]$.
\end{center}

\begin{tikzpicture}[scale=1.3]
\draw[blue] (-0.5,0) -- ++(-1.5,0.866) -- ++(0,-1.732) -- cycle;
\draw[blue] (0.5,0) -- ++(1.5,0.866) -- ++(0,-1.732) -- cycle;
\draw[red] (-0.5,0) -- (0.5,0);
\draw[red] (-2,0.866) -- (2,0.866);
\draw[red] (-2,-0.866) -- (2,-0.866);
\node at (-0.3,0.2) {$x+k$};
\node at (0.3,-0.2) {$-x-3k$};
\node at (-2.25,0.866) {$x$};
\node at (-2.55,-0.866) {$x+2k$};
\node at (2.65,0.866) {$-x-2k$};
\node at (2.6,-0.866) {$-x-k$};
\fill[black] (0.5,0) circle (2pt);
\fill[black] (-0.5,0) circle (2pt);
\fill[black] (-2,0.866) circle (2pt);
\fill[black] (-2,-0.866) circle (2pt);
\fill[black] (2,0.866) circle (2pt);
\fill[black] (2,-0.866) circle (2pt);
\end{tikzpicture}

\begin{center}
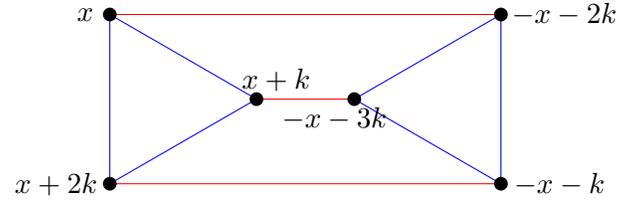
\textbf{Figure 6 - } The component of $x\in [3k+1,3k+i-1]$.
\end{center}

\end{multicols}

\begin{multicols}{2}
\begin{center}
\begin{footnotesize}
\begin{tikzpicture}[scale=1.3]
\draw[blue] (0,0) -- (-1,1.732) -- ++(-1,-1.732);
\draw[thick,blue] (0,0) -- (-2,0);
\draw[thick,dashed,red] (0,0) -- (-2,0);
\draw[red] (-1,1.732) arc (-90:-480:0.25cm);
\fill[black] (0,0) circle (2pt);
\fill[black] (-1,1.732) circle (2pt);
\fill[black] (-2,0) circle (2pt);
\node at (0,-0.2) {$4k+i+\frac{k-i}{2}$};
\node at (-2,-0.2) {$5k+i+\frac{k-i}{2}$};
\node at (-0.05,1.732) {$3k+i+\frac{k-i}{2}$};
\end{tikzpicture}
\end{footnotesize}
\end{center}

\begin{center}
\textbf{Figure 7 - } The component of $3k+i+\frac{k-i}{2}$.
\end{center}

\begin{center}
\begin{footnotesize}
\begin{tikzpicture}[scale=1.3]
\draw[blue] (0,0) -- (-1,1.732) -- ++(-1,-1.732);
\draw[thick,blue] (0,0) -- (-2,0);
\draw[thick,dashed,red] (0,0) -- (-2,0);
\draw[red] (-1,1.732) arc (-90:-480:0.25cm);
\fill[black] (0,0) circle (2pt);
\fill[black] (-1,1.732) circle (2pt);
\fill[black] (-2,0) circle (2pt);
\node at (0,-0.2) {$3k+i/2$};
\node at (-2,-0.2) {$4k+i/2$};
\node at (-0.25,1.732) {$5k+i/2$};
\end{tikzpicture}
\end{footnotesize}
\end{center}

\begin{center}
\textbf{Figure 8 - } The component of $3k+i/2$.
\end{center}

\end{multicols}

\begin{center}
Type 1 edges are represented in  blue and type 2 edges in red. Edges that are both type 1 and type 2 are represented by red and blue dashed lines.
\end{center}

Let us now assume $i=0$. 
When $k$ is even, $\Gamma$ contains  $k/2-1$ disjoint copies of $K_2\square K_3$;
indeed the component of every $x \in [3k+1,4k-1]\setminus \{3k+k/2\}$ is as in Figure~5.
The other two components of $\Gamma$ are: a looped triangle $\{3k+k/2,4k+k/2,5k+k/2\}$
with a type 2 loop at $3k+k/2$, a type 2 edge between $4k+k/2$ and $5k+k/2$ and all possible (non-loop) type 1 edges; 
a triangle $\{4k,5k,6k\}$ with a loop at $5k$ (which is both bad and type 2) and a bad loop at $6k$, where all non-loop edges of the triangle are type 1, except that additionally the edge between $4k$ and $6k$ is also type 2.
Note that $\mis (\Gamma)=2 \cdot 6^{k/2-1}$, $\mis (\Gamma ')=4 \cdot 6^{k/2-1}$ 
and $\mis(\Gamma_1\rtimes\Gamma_2)=6^{k}$.
Then by
Lemma~\ref{generalisation link graph},
\begin{align*}
    \mis(\Tilde{\Gamma})&=2 \cdot 6^{k/2-1}\cdot4 ^{a-1} \cdot6^{(k/2-1)(a-1)}\cdot6^{\frac{k}{2}(|K|-a)}\\
    &=\frac{1}{2} \cdot (2/3)^{a}\cdot  6^{\frac{k}{2}|K|}
    \geq \frac{1}{2} \cdot (2/3)^{|K|}\cdot6^{\frac{m}{18}|K|}
    \geq  (2/3)^{n/m+1}\cdot6^{(\frac{1}{18}-\frac{4}{9m})n}.
\end{align*}

When $k$ is odd, $\Gamma$ consists of $(k-1)/2$ copies of $K_2\square K_3$ and the triangle $\{4k,5k,6k\}$ with two loops as in the $k$ even case.  Therefore, 
$\mis (\Gamma)= 6^{(k-1)/2}$, $\mis (\Gamma ')= 2 \cdot 6^{(k-1)/2}$ 
and $\mis(\Gamma_1\rtimes\Gamma_2)=6^{k}$. Each of these terms are at least as big as the corresponding terms in the $k$ even case, so again by Lemma~\ref{generalisation link graph} we conclude that $\mis(\Tilde{\Gamma}) \geq (2/3)^{n/m+1}\cdot6^{(\frac{1}{18}-\frac{4}{9m})n}$. 
Therefore, arguing as before, in both cases $\fm(G)\geq \mis(\Tilde{\Gamma})\geq (2/3)^{1+n/m}\cdot 6^{(\frac{1}{18}-\frac{4}{9m})n}$, as desired.
\end{proof}

This lemma now easily allows us to confirm the lower bound in Conjecture \ref{conj1} for groups which have a sufficiently large cyclic component.

\begin{proposition}\label{prop34}
Let $G$ be an $n$-order abelian group. If there exists $m\geq 3084$ and an abelian group $K$ such that $G=\Z_m\times K$, then $f_{\max}(G)\geq 2^{n/7}$.
\end{proposition}

\begin{proof}
Provided $n\geq m\geq 3084$, then by Lemma \ref{large cyclic component},
$$f_{\max}(G)\geq (2/3)^{1+n/m}\cdot 6^{(\frac{1}{18}-\frac{4}{9m})n}\geq 2^{n/7}.$$
\end{proof}

\subsection{Some type \rom{3} groups} 
Amongst all abelian groups of order $n$, 
Theorem~\ref{mu(G)} tells us that the smallest values of $\mu (G) $ are obtained by type \rom{3} groups. So one might think they are the most likely groups to disprove the lower bound in Conjecture~\ref{conj1}. However, in this subsection we give lower bound constructions for some type \rom{3} groups.

\begin{proposition}
Let $G$ be a type \rom{3} group of order $n$ and let $m$ be its exponent. If $m\in\{7,13,19\}$, then $\fm(G)\geq 2^{\mu(G)/2-2}\geq 2^{n/7-2}$.
\end{proposition}

\begin{proof}
If $m=7$ then $G=\Z_7^k$; Proposition 5.7 in \cite{BLST2} gives that $\fm(G)\geq  2^{\mu({G})/2-1}=2^{n/7-1}$ in this case. If $m=13$, let $H<G$ be a subgroup of $G$ of index 13 and $x\in 3+H$. Let $T:=(1+H)\cup(4+H)\cup(6+H)\cup(9+H)$ be a sum-free set of size $4n/13$ and consider the link graph $\Gamma:=L_{\{x\}}[T]$. Then there is a loop at $2x\in 6+H$, and the remaining edges of $\Gamma$ form a perfect matching between the elements of $1+H$ and those of $4+H$ and between those of $6+H$ and those of $9+H$. Thus $\mis(\Gamma)=2^{2n/13-1}$ and $\fm(G)\geq 2^{2n/13-1}= 2^{\mu(G)/2-1}$. If $m=19$ then let $H<G$ be a subgroup of index 19 and $x\in 6+H$. By considering the sum-free set $T:=(1+H)\cup(3+H)\cup(12+H)\cup(14+H)\cup(16+H)\cup(18+H)$ and the link graph $\Gamma:=L_{\{x\}}[T]$ we obtain that $\mis(\Gamma)=2^{3n/19-2}$ (there is a loop at $2x\in 12+H$ and another one at the unique $y\in 3+H$ such that $2y=x$), thus, $\fm(G)\geq2^{3n/19-2}=2^{\mu(G)/2-2}$.
\end{proof}

\section{Concluding remarks and open problems}\label{sec:con}
\subsection{Structural results and sharp  bounds on $\fm (G)$}
In this paper we have provided a sharp count on the number of maximal sum-free sets in both $\mathbb Z^k_2$ and $\mathbb Z^k_3$. Crucial ingredients of both these proofs were structural results for `large' sum-free subsets of $\mathbb Z^k_2$~\cite{ped, dav} and $\mathbb Z^k_3$~\cite{lev}. 
To obtain sharp bounds on $\fm (G)$ in general, further such 
 structural results will be needed.
Lemma~5.6 in~\cite{GR-g} provides a structural result for large sum-free subsets of
type~\rom{1} groups (and was applied to obtain bounds on $\fm (G)$
in~\cite{ls}). 
However, our understanding of 
 how $\fm (G)$ should behave in general is still rather limited. Thus, it is unclear how precise structural results need to be for applications to the $\fm (G)$ problem.
Perhaps the next natural question to consider is to determine the asymptotic behaviour of $\fm (\mathbb Z^k_p)$ when $p \geq 5$ is prime.

\subsection{Counting complete caps in projective spaces}\label{sec:pro}
Let $PG(n,q)$ be the projective space of dimension $n$ over the Galois field $\mathbb F_q$.
An \emph{$\ell$-cap} is a set of $\ell$ points no three of which are collinear. An $\ell$-cap is said to be \emph{complete} if it is not contained in an $(\ell + 1)$-cap. Complete caps are well-studied objects, for example, the size of the smallest complete cap in a projective space (see e.g.~\cite{kimvu}) and the spectrum of values of $\ell$ for which there is a complete $\ell$-cap (see e.g.~\cite{dav2}).

In this setting note that a maximal sum-free set in $\mathbb Z^{k+1}_2$ is precisely a complete cap in $PG(k,2)$. Thus, Theorem~\ref{main} asymptotically determines the number of complete caps in $PG(k,2)$. It would be interesting to obtain analogous counting results for other projective spaces. Furthermore, it would be interesting to attack other open problems on complete caps in $PG(k,2)$ (such as those in~\cite[Section 8]{dav2}) by  instead working in the setting of maximal sum-free sets in $\mathbb Z^{k+1}_2$.

\subsection{Maximal distinct sum-free sets}\label{sec:open}

Let $G$ be an abelian group of order $n$. We call a triple $\{x,y,z\}\subseteq G$ a \textit{distinct Schur triple} if it is a Schur triple and its three elements are distinct. A subset $S\subseteq G$ is a \textit{distinct sum-free set} if it does not contain any distinct Schur triple. A distinct sum-free set $S$ is a \textit{maximal distinct sum-free set} if $S$ is not properly contained in another distinct sum-free subset of $G$. Denote by $\mu^{\star}(G)$ the size of a largest distinct sum-free subset of $G$. We write $f^{\star}(G)$ for the number of distinct sum-free subsets of $G$ and $\fm^{\star}(G)$ for the number of maximal distinct sum-free subsets of $G$. Note that an important feature of distinct sum-free sets in comparison with sum-free sets is that they can contain the zero element. Moreover, we immediately have that a sum-free set is a distinct sum-free set. Also, if $A\subseteq G$ is a distinct sum-free subset of $G$, it can be written as $A=B\cup C$ with $B$ sum-free and $|C|=o(n)$. Indeed, $A\backslash\{0\}$ (or $A$ if $0\notin A$) has at most $n-1=o(n^2)$ Schur triples, because such a Schur triple is necessarily of the form $\{x,x,2x\}$ and there are $|A\backslash\{0\}|\leq n-1$ choices for $x\in A\backslash\{0\}$. Therefore, by  Green's removal lemma~\cite[Theorem 1.4]{G-R}, $A\backslash\{0\}$ can be made sum-free by removing $o(n)$ elements, and so $A$ can be made sum-free by removing at most $o(n)+1=o(n)$ elements. It follows from this observation that
\begin{equation}\label{eqa}
   \mu(G)\leq \mu^{\star}(G)\leq\mu(G)+o(n) 
\end{equation}
and 
\begin{equation}\label{eqb}
    f(G)\leq f^{\star}(G)\leq 2^{o(n)}f(G).
\end{equation}

It is natural to guess that $ \fm(G)\leq \fm ^{\star}(G)$ for all finite abelian groups $G$. However, this seems more challenging to prove than (\ref{eqa}) and (\ref{eqb}).
 Indeed, a maximal sum-free set is not necessarily a maximal distinct sum-free set, and we can even have two maximal sum-free sets that lie in the same maximal distinct sum-free set.

\begin{example}
In $G=\Z_7$, $\{2,3\}$ and $\{3,4\}$ are two maximal sum-free sets that lie in the same maximal distinct sum-free set $\{2,3,4\}$. However, we still have $\fm(\Z_7)=9<\fm^\star(\Z_7)=14$.
\end{example}

This observation motivates the following question.

\begin{question}\label{exception}
Are there some abelian groups $G$ for which $\fm^{\star}(G)<\fm(G)$?
\end{question}

From  the perspective that a group answering  Question~\ref{exception} positively  would perhaps be an exception, one can wonder in general if $\fm(G)$ is much smaller than $\fm^{\star}(G)$.

\begin{question}\label{exponential gap}
For which abelian groups $G$ is $\fm(G)$ exponentially smaller than  $\fm^{\star}(G)$?
\end{question}

We will demonstrate examples of abelian groups $G$ for which the behaviour in 
 Question~\ref{exponential gap} does occur by giving a lower bound on $\fm^{\star}(G)$. For this we need to modify our definition of the link graph to fit  the concept of distinct sum-free sets. For subsets $B,S\subseteq G$ let $L^{\star}_S[B]$ be the \textit{distinct link graph of $S$ on $B$} defined as follows. Its vertex set is $B$ and its edge set consists of the following edges:

\begin{enumerate}[label = (\roman*),itemsep=0pt]
    \item two distinct vertices $x,y\in B$ are adjacent if there exists $s\in S$ such that $\{x,y,s\}$ is a distinct Schur triple;
    \item there is a loop at a vertex $x\in B$ if there exist distinct $s,s'\in S$  such that $\{x,s,s'\}$ is a distinct Schur triple.
\end{enumerate}

We will use distinct link graphs to provide lower bound constructions on $\fm^{\star}(G)$ for finite abelian groups $G$. Speaking of which, note that if $B,S\subseteq G$ are disjoint distinct sum-free subsets of $G$, we have that any two  maximal independent sets $I, I'$ in $L^{\star}_S[B]$ 
are such that $I \cup S$ and $I' \cup S$
lie in two different maximal distinct sum-free subsets of $G$, because no more elements from $B$ can be added to either without destroying the distinct sum-free property. Thus, for such $B$ and $S$ we have
\begin{equation}\label{lower bound fm star link graph}
    \fm^{\star}(G)\geq \mis(L_S^{\star}[B]).
\end{equation}

\begin{proposition}\label{lower bound fm star}
For any finite abelian group $G$ which is not of the form $G=\Z_2^k\times K$ with $k\geq 1$ and $|K|\geq 3$ odd, we have $\fm^{\star}(G)\geq 2^{\frac{\mu(G)-1}{2}}$.
\end{proposition}

\begin{proof}

We first assume that $|G|$ is odd. If $G$ is type \rom{1}$(p)$, with $p\ne2$ and $p\equiv2\pmod{3}$, then by Theorem \ref{mu(G)}, $\mu(G)=|G|\left(\frac{1}{3}+\frac{1}{3p}\right)$. Let $H<G$ be a subgroup of index $p$ and $g\notin H$ such that the cosets of $H$ are $H$, $g+H$,\ldots,$(p-1)g+H$. As in \cite[Theorem 2]{yan} we consider $T:=(g+H)\cup(4g+H)\cup\ldots\cup ((p-1)g+H)$. Note that $T$ is sum-free and has size $|T|=\frac{p+1}{3}|H|=\mu(G)$. Then $\Gamma:=L_{\{0\}}^{\star}[T]$ matches perfectly the elements of $ig+H$ with those of $-ig+H$, for $i=1,4,\ldots,p-1$;  since $p$ is an odd prime number congruent to $2\pmod{3}$, the number of cosets $(p+1)/{3}$ is even. 
Thus, $\Gamma$ is a perfect matching and
so by (\ref{lower bound fm star link graph}), $\fm^{\star}(G)\geq \mis(\Gamma)=2^{{\mu(G)}/{2}}$. 

If $G$ is type \rom{2}, by Theorem \ref{mu(G)}, $\mu(G)=|G|/{3}$. We give a construction similar to \cite[Proposition 5.6]{BLST2}. Let $H<G$ be a subgroup of index 3. Then there are three cosets: $H$, $1+H$ and $2+H$. Let $T:=1+H$, $x\in 2+H$ and $\Gamma:=L_{\{x\}}^{\star}[T]$.  Note there is no loop at $2x\in T$ because $\{x,x,2x\}$ is not a distinct Schur triple. Further, in the case where $y_0=x-y_0$ (which happens for a unique $y_0\in T$ since $|G|$ is odd), $y_0$ is an isolated vertex. 
All other $y\in 1+H$ have a unique neighbour in $\Gamma$, namely $x-y$.
Then by (\ref{lower bound fm star link graph}), $\fm^{\star}(G)\geq \mis(\Gamma)=2^{(|T|-1)/{2}}=2^{(\mu(G)-1)/{2}}$. 

If $G$ is type \rom{3}, then by Theorem \ref{mu(G)}, $\mu(G)=|G|\left(\frac{1}{3}+\frac{1}{3m}\right)$ where $m$ is the largest order of any element of $G$. Let $g\in G$ be an element of order $m$ and $H<G$ a subgroup of index $m$ such that its cosets are $H$,$g+H$,\ldots,$(m-1)g+H$. 
Note that as $G$ is type~\rom{3}, $m$ is odd and $m \equiv 1 \pmod{3}$.
As in \cite[Theorem 5]{yan} we consider $T:=(2g+H)\cup(5g+H)\cup\ldots\cup((m-2)g+H)$. Note that $T$ is sum-free and has size $|T|=\frac{m-1}{3}|H|=\mu(G)$. Similarly to the case of type~\rom{1} groups, $\Gamma:=L_{\{0\}}^{\star}[T]$ is a perfect matching. Thus by (\ref{lower bound fm star link graph}), $\fm^{\star}(G)\geq \mis(\Gamma)=2^{{\mu(G)}/{2}}$.

\smallskip

Now we suppose that $|G|$ is even. In this case, by Theorem~\ref{mu(G)}, $\mu(G)={|G|}/{2}$. Since $G$ is not of the form $\Z_2^k\times K$ with $|K|\geq 3$ odd, either $G=\Z_{2^\alpha}\times K$ for some $K$ and $\alpha\geq2$, or $G=\Z_2^k$. In the first case, consider $T:=\{1,3,\ldots,2^\alpha-1\}\times K$. Then $\Gamma:=L_{\{0\}}^{\star}[T]$ is a perfect matching between the elements $(a,k)$ and $(-a,-k)$ with $a\in\{1,3,\ldots,2^\alpha-1\}$, $k\in K$. Since $\alpha\geq 2$, $a\ne-a$ for all $a\in\{1,3,\ldots,2^\alpha-1\}$. Thus by (\ref{lower bound fm star link graph}), $\fm^{\star}(G)\geq\mis(\Gamma)=2^{{|T|}/{2}}=2^{{\mu(G)}/{2}}$. When $G=\Z_2^k$, note there is a one-to-one correspondence between the maximal sum-free sets of $G$ and the maximal distinct sum-free sets of $G$. Indeed, since $2x=0$ for all $x\in G$, a Schur triple which does not contain the zero element is a distinct Schur triple. Conversely, a distinct Schur triple cannot contain the zero element. Thus, $A\mapsto A\cup\{0\}$ is a bijection between  the maximal sum-free subsets of $G$ and the maximal distinct sum-free subsets of $G$. Hence $\fm^{\star}(\Z_2^k)=\fm(\Z_2^k)\geq 2^{{\mu(\Z_2^k)}/{2}}$, where the latter inequality was proven in~\cite[Proposition 5.3]{BLST2}.
\end{proof}

\begin{proposition}\label{lower bound fm star exceptions}
Let $G$ be an abelian group of order $n$ of the form $G=\Z_2^k\times K$ with $k\geq 1$ and $|K|\geq 3$ odd and let $a:=2^{k-1}$. Then $\fm^{\star}(G)\geq 2^{(\frac{1}{2}-\frac{a}{n})\mu(G)}$.
\end{proposition}

\begin{proof}
Recall that by Theorem \ref{mu(G)} we know  $\mu(G)=n/2$. Let $q\geq 3$ be an odd number and $K'$ be a group such that $q||K|$ and  $K=\Z_q\times K'$. Then $G$ can be written as $G=\Z_m\times G'$ with $m:=2q$ and $G':=\Z_2^{k-1}\times K'$. We consider $T:=\{1,3,\ldots,m-1\}\times G'$ which is sum-free, and $\Gamma:=L_{\{0\}}^\star[T]$. Then each element of $T$ of the form $(b,g')$ with $b\in\{1,3,\ldots,m-1\}\backslash\{m/2\}$ and $g'\in G'$ has a unique neighbour in $\Gamma$ distinct from themselves, namely $(-b,-g')$. Since $m/2$ is odd, for all $g'\in G'$, $(m/2,g')\in T$. So if $g'\in G'$ is such that $g'\ne -g'$, then $(m/2,g')$ also has  a unique neighbour in $\Gamma$ different from themselves; otherwise it is an isolated vertex. As $G'$ has $a$ elements $g'$ such that $g'=-g'$,
$$\mis(\Gamma)=2^{\frac{m/2-1}{2}|G'|+\frac{|G'|-a}{2}}=2^{\frac{n}{4}-\frac{a}{2}}=2^{(\frac{1}{2}-\frac{a}{n})\mu(G)},$$
and by (\ref{lower bound fm star link graph}), $\fm^\star(G)\geq 2^{(\frac{1}{2}-\frac{a}{n})\mu(G)}$.
\end{proof}

In \cite[Theorem 3.2]{ls}, Liu and Sharifzadeh proved that there is a constant $c>10^{-4}$ such that for all even order abelian groups $G$ that have a negligible number of order-2 elements (e.g. $a=o(n)$ in Proposition \ref{lower bound fm star exceptions}) we have $\fm(G)<2^{(1/2-c)\mu(G)}$ for sufficiently large $|G|$. This result combined with Propositions \ref{lower bound fm star} and \ref{lower bound fm star exceptions} gives a positive answer to Question \ref{exponential gap} for those abelian groups. We also know there are some groups for which the answer to this question is negative. Indeed, we saw in the proof of Proposition \ref{lower bound fm star} that $\fm^{\star}(\Z_2^k)=\fm(\Z_2^k)$.

\section*{Acknowledgment}
The authors are grateful to the referee for their helpful and careful review.

\end{document}